\documentclass{article}
\usepackage{amsfonts}
\usepackage{amsmath}

\newtheorem{theorem}{Theorem}

\newtheorem{lemma}[theorem]{Lemma}

\newenvironment{proof}[1][Proof]{\noindent\textbf{#1.} }{\ \rule{0.5em}{0.5em}}
\input{tcilatex}
\begin{document}

\title{The Regularized Trace of Sturm-Liouville Problem with Discontinuities
at Two Points}
\author{F. H\i ra and N. Alt\i n\i \c{s}\i k \\
%EndAName
Department of Mathematics,Arts and Science  Faculty, \\
Ondokuz May\i s University, 55139, Samsun, Turkey\\
email: fatmahira@yahoo.com.tr}
\maketitle

\begin{abstract}
In this paper, we obtain a regularized trace formula for a Sturm Liouville
problem which has two points of discontinuity and also contains an
eigenparameter in a boundary ondition.
\end{abstract}

\section{Introduction}

The regularized trace for the classical Sturm Liouville problem was first
calculated by Gelfand and Levitan [19]. Then this work was continued by many
authors ( see [1], [5-8], [12-18] and [20],respectively). A regularized
trace formula for Sturm-Liouville equation with one or two boundary
conditions depending on a spectral parameter was investigated in
[2,4,6,11,22,24]. The regularized trace formula of the infinite sequence of
eigenvalues for some version of a Dirichlet boundary value problem with
turning points was calculated in [9]. All of these works are on the traces
of continuous boundary value problems.

Consider the boundary value problem%
\begin{equation}
\tau \left( u\right) :=-u^{\prime \prime }+q\left( x\right) u=\lambda u,%
\text{ ~}x\in I,  \tag{1.1}
\end{equation}%
with boundary conditions%
\begin{equation}
B_{a}\left( u\right) :=u^{\prime }\left( a\right) =0,  \tag{1.2}
\end{equation}%
\begin{equation}
B_{b}\left( u\right) :=\left( \lambda -h\right) u\left( b\right) -u^{\prime
}\left( b\right) =0,  \tag{1.3}
\end{equation}%
and transmission conditions at two points of discontinuity%
\begin{equation}
T_{c_{1}}\left( u\right) :=\binom{u\left( c_{1}+\right) }{u^{\prime }\left(
c_{1}+\right) }-\frac{1}{\delta }\binom{u\left( c_{1}-\right) }{u^{\prime
}\left( c_{1}-\right) }=0,  \tag{1.4}
\end{equation}%
\begin{equation}
T_{c_{2}}\left( u\right) :=\binom{u\left( c_{2}-\right) }{u^{\prime }\left(
c_{2}-\right) }-\frac{\gamma }{\delta }\binom{u\left( c_{2}+\right) }{%
u^{\prime }\left( c_{2}+\right) }=0,  \tag{1.5}
\end{equation}%
where$~I:=\left[ a,c_{1}\right) \cup \left( c_{1},c_{2}\right) \cup \left(
c_{2},b\right] ,~\lambda $ is a spectral parameter, $q\left( x\right) $ is a
real valued function which is continuous in$~\left[ a,c_{1}\right) ,$ $%
\left( c_{1},c_{2}\right) $ and $\left( c_{2},b\right] $ and has finite
limits $q\left( c_{1}\pm \right) :=\lim_{x\rightarrow c_{1}\pm }q\left(
x\right) ,$ $q\left( c_{2}\pm \right) :=\lim_{x\rightarrow c_{2}\pm }q\left(
x\right) $, $h,~\delta ,~\gamma $ are real numbers $\delta \neq 0,$ $\gamma
\neq 0.$

As far as we know, there are two works about the regularized trace of
discontinuous boundary value problems [21,23]. They are investigated by
Yang. In [21] and [23], the author studied regularized sums from the
eigenvalues, oscillations of eigenfunctions and the solutions of inverse
nodal problem for a discontinuous boundary value problem with retarded
argument and obtained some formulas for the regularized traces of
second-order differential operators with discontinuities inside a finite
interval, respectively. In these works, the problem has one point of
discontinuity and not contain a spectral parameter in boundary conditions.

The aim of the present paper is to obtain a formula for the regularized
trace of the problem (1.1)-(1.5). The problem (1.1)-(1.5) is a Sturm
Liouville problem which has two points of discontinuity and the boundary
conditions depending on an eigenparameter. The regularized trace of Sturm
Liouville problem with two discontinuities has not been studied before. We
follow the method for computing the traces of the problem (1.1)-(1.5) in
[16,17]. Firstly, we give some preliminaries for asmptotic formulas of
solution and eigenvalues. Then, we proved the regularized first trace
formula for the problem (1.1)-(1.5).

\section{Preliminaries}

Let us denote the solution of equation (1.1) by%
\begin{equation}
\phi \left( x,\lambda \right) =\left\{ 
\begin{array}{c}
\phi _{1}\left( x,\lambda \right) ,\text{ \ \ \ }x\in \lbrack a,c_{1}), \\ 
\phi _{2}\left( x,\lambda \right) ,\text{\ }x\in \left( c_{1},c_{2}\right) ,
\\ 
\phi _{3}\left( x,\lambda \right) ,\text{ \ \ \ }x\in \left( c_{2},b\right] ,%
\end{array}%
\right.   \tag{2.1}
\end{equation}%
satisfying the initial conditions%
\begin{equation}
\phi _{1}\left( a,\lambda \right) =1,\ \phi _{1}^{\prime }\left( a,\lambda
\right) =0,  \tag{2.2}
\end{equation}%
\begin{equation}
\phi _{2}\left( c_{1},\lambda \right) =\delta ^{-1}\phi _{1}\left(
c_{1}-,\lambda \right) ,\text{ \ }\phi _{2}^{\prime }\left( c_{1},\lambda
\right) =\delta ^{-1}\phi _{1}^{\prime }\left( c_{1}-,\lambda \right) , 
\tag{2.3}
\end{equation}%
\begin{equation}
\phi _{3}\left( c_{2},\lambda \right) =\delta \gamma ^{-1}\phi _{2}\left(
c_{2}-,\lambda \right) ,\text{ \ }\phi _{3}^{\prime }\left( \theta
_{+\varepsilon }\right) =\delta \gamma ^{-1}\phi _{2}^{\prime }\left(
c_{2}-,\lambda \right) .  \tag{2.4}
\end{equation}%
It is obvious that$~\phi \left( x,\lambda \right) $ satisfies the equation
(1.1) on $I,$~the boundary condition (1.2) and the transmission conditions
(1.4) and (1.5).

The asymptotics formulas of the eigenvalues and eigenfunctions can be
derived similar to the classical techniques of $\left[ 3,5,13,19\right] .$
We state the results briefly. Denote $\lambda =s^{2}.~$The solution of
equation (1.1), fulfilling the conditions (2.2)-(2.4), satisfies the
integral equation%
\begin{eqnarray}
\phi _{3}\left( x,\lambda \right) &=&\frac{1}{\gamma }\cos \left( s\left(
x-a\right) \right) +\frac{1}{2s\gamma }\left( Q_{1}\left( c_{1}\right)
+Q_{2}\left( c_{2}\right) \right) \sin \left( s\left( x-a\right) \right) + 
\notag \\
&&\frac{1}{4s^{2}\gamma }\left\{ q\left( c_{2}\right) \cos \left( s\left(
x-2c_{2}+a\right) \right) -\left( q\left( a\right) +Q_{1}\left( c_{1}\right)
Q_{2}\left( c_{2}\right) \right) \times \right. +  \notag \\
&&\left. \cos \left( s\left( x-a\right) \right) \right\} +\frac{1}{%
8s^{3}\gamma }\left\{ q\left( c_{1}\right) \left( Q_{1}\left( c_{1}\right)
+Q_{2}\left( c_{2}\right) \right) \sin \left( s\left( x-2c_{1}+a\right)
\right) -\right.  \notag \\
&&\left( q\left( c_{2}\right) Q_{1}\left( c_{1}\right) -q^{\prime }\left(
c_{2}\right) \right) \sin \left( s\left( x-2c_{2}+a\right) \right) -\left(
q^{\prime }\left( a\right) +q\left( a\right) Q_{2}\left( c_{2}\right)
\right) \times +  \notag \\
&&\left. \sin \left( s\left( x-a\right) \right) \right\} +\dfrac{1}{s}%
\underset{c_{2}}{\overset{x}{\int }}\sin \left( s\left( x-y\right) \right)
q\left( y\right) \phi _{3}\left( y,\lambda \right) dy,  \TCItag{2.5}
\end{eqnarray}%
where%
\begin{equation}
Q_{1}\left( x\right) =\underset{a}{\overset{x}{\int }}q\left( y\right) dy,~%
\text{\ \ \ }Q_{2}\left( x\right) =\underset{c_{1}}{\overset{x}{\int }}%
q\left( y\right) dy.  \tag{2.6}
\end{equation}

Solving the equation (2.5) by the method of successive approximations, we
obtain

\begin{eqnarray}
\phi _{3}\left( x,\lambda \right) &=&\frac{1}{\gamma }\cos \left( s\left(
x-a\right) \right) +\frac{\alpha _{1}\left( x\right) }{s\gamma }\sin \left(
s\left( x-a\right) \right) +  \notag \\
&&\frac{\alpha _{2}\left( x\right) }{s^{2}\gamma }\cos \left( s\left(
x-a\right) \right) +\frac{1}{s^{3}\gamma }\left\{ \alpha _{3}\left( x\right)
\sin \left( s\left( x-a\right) \right) +\right.  \notag \\
&&\left. \alpha _{4}\left( x\right) \sin \left( s\left( x-2c_{2}+a\right)
\right) \right\} +O\left( \frac{1}{s^{4}}\right) ,  \TCItag{2.7}
\end{eqnarray}

\begin{eqnarray}
\phi _{3}^{\prime }\left( x,\lambda \right) &=&-\frac{s}{\gamma }\sin \left(
s\left( x-a\right) \right) +\frac{\alpha _{1}\left( x\right) }{\gamma }\cos
\left( s\left( x-a\right) \right) +\frac{1}{s\gamma }\left( \alpha
_{1}^{\prime }\left( x\right) -\alpha _{2}\left( x\right) \right) \times 
\notag \\
&&\sin \left( s\left( x-a\right) \right) +\frac{1}{s^{2}\gamma }\left\{
\left( \alpha _{3}\left( x\right) +\alpha _{3}^{\prime }\left( x\right)
\right) \cos \left( s\left( x-a\right) \right) +\right.  \notag \\
&&\left. \alpha _{4}\left( x\right) \cos \left( s\left( x-2c_{2}+a\right)
\right) \right\} +\frac{1}{s^{3}\gamma }\left\{ \alpha _{3}^{\prime }\left(
x\right) \sin \left( s\left( x-a\right) \right) +\right.  \notag \\
&&\left. \alpha _{4}^{\prime }\left( x\right) \sin \left( s\left(
x-2c_{2}+a\right) \right) \right\} +O\left( \frac{1}{s^{4}}\right) , 
\TCItag{2.8}
\end{eqnarray}%
where%
\begin{eqnarray}
\alpha _{1}\left( x\right) &=&\frac{1}{2}\left( Q_{1}\left( c_{1}\right)
+Q_{2}\left( c_{2}\right) +Q_{3}\left( x\right) \right) ,  \notag \\
\alpha _{2}\left( x\right) &=&\frac{1}{4}\left( q\left( x\right) -q\left(
a\right) -Q_{1}\left( c_{1}\right) Q_{2}\left( c_{2}\right) -Q_{3}\left(
x\right) \left( Q_{1}\left( c_{1}\right) +Q_{2}\left( c_{2}\right) \right)
\right) ,  \notag \\
\alpha _{3}\left( x\right) &=&\frac{1}{8}\left( q\left( x\right) \left(
Q_{1}\left( c_{1}\right) +Q_{2}\left( c_{2}\right) \right) -q^{\prime
}\left( x\right) -q^{\prime }\left( a\right) -q\left( a\right) Q_{2}\left(
c_{2}\right) +\right.  \notag \\
&&q\left( c_{1}\right) \left( Q_{1}\left( c_{1}\right) +Q_{2}\left(
c_{2}\right) \right) \left. -Q_{3}\left( x\right) \left( q\left( a\right)
+Q_{1}\left( c_{1}\right) Q_{2}\left( c_{2}\right) \right) \right) ,  \notag
\\
\alpha _{4}\left( x\right) &=&\frac{1}{4}\left( q\left( c_{2}\right)
Q_{1}\left( c_{1}\right) -q^{\prime }\left( c_{2}\right) \right) +\frac{1}{8}%
q\left( c_{2}\right) \left( Q_{3}\left( x\right) +Q_{2}\left( c_{2}\right)
\right) ,  \notag \\
Q_{3}\left( x\right) &=&\underset{c_{2}}{\overset{x}{\int }}q\left( y\right)
dy.  \TCItag{2.9}
\end{eqnarray}

It is obvious that the characteristic function $\omega \left( \lambda
\right) $ of the problem (1.1)-(1.5) is as follows%
\begin{equation}
\omega \left( \lambda \right) =\left( \lambda -h\right) \phi _{3}\left(
b,\lambda \right) -\phi _{3}^{\prime }\left( b,\lambda \right) ,  \tag{2.10}
\end{equation}%
and the eigenvalues of the problem (1.1)-(1.5) coincide with the roots of $%
\omega \left( \lambda \right) .$ From (2.7) and (2.8) we have%
\begin{eqnarray}
\omega \left( \lambda \right) &=&\frac{s^{2}}{\gamma }\cos \left( s\left(
b-a\right) \right) +\frac{s}{\gamma }\left( 1+\alpha _{1}\left( b\right)
\right) \sin \left( s\left( b-a\right) \right) +  \notag \\
&&\frac{1}{\gamma }\left( \alpha _{2}\left( b\right) -\alpha _{1}\left(
b\right) -h\right) \cos \left( s\left( b-a\right) \right) +\frac{1}{s\gamma }%
\left\{ \left( \alpha _{3}\left( b\right) -h\alpha _{1}\left( b\right)
\right. +\right.  \notag \\
&&\left. \left. \alpha _{2}\left( b\right) -\alpha _{1}^{\prime }\left(
b\right) \right) \sin \left( s\left( b-a\right) \right) +\alpha _{4}\left(
b\right) \sin \left( s\left( b-2c_{2}+a\right) \right) \right\} +  \notag \\
&&O\left( \frac{1}{s^{2}}\right) .  \TCItag{2.11}
\end{eqnarray}

By the Rouche theorem and from the asymptotic formula (2.11), we obtain%
\begin{equation}
s_{n}=\frac{\left( n-1/2\right) \pi }{b-a}+\frac{1}{\left( n-1/2\right) \pi }%
\left( 1+\alpha _{1}\left( b\right) \right) +O\left( \frac{1}{n^{2}}\right) .
\tag{2.12}
\end{equation}

It follows from (2.12) that%
\begin{equation}
\lambda _{n}=\left( \frac{\left( n-1/2\right) \pi }{b-a}\right)
^{2}+K+O\left( \frac{1}{n^{2}}\right) ,  \tag{2.13}
\end{equation}%
where 
\begin{equation}
K=\frac{2}{b-a}\left( 1+\alpha _{1}\left( b\right) \right) ,  \tag{2.14}
\end{equation}%
\begin{equation}
\alpha _{1}\left( b\right) =\frac{1}{2}\left( Q_{1}\left( c_{1}\right)
+Q_{2}\left( c_{2}\right) +Q_{3}\left( b\right) \right) ,  \tag{2.15}
\end{equation}%
\begin{equation}
Q_{1}\left( c_{1}\right) =\underset{a}{\overset{c_{1}}{\int }}q\left(
y\right) dy,\text{ \ }Q_{2}\left( c_{2}\right) =\underset{c_{1}}{\overset{%
c_{2}}{\int }}q\left( y\right) dy,\text{ \ \ }Q_{3}\left( b\right) =\underset%
{c_{2}}{\overset{b}{\int }}q\left( y\right) dy.  \tag{2.16}
\end{equation}

\section{Traces of The Problem}

The series%
\begin{equation}
\lambda _{0}+\sum\limits_{n=1}^{\infty }\left( \lambda _{n}-\left( \frac{%
\left( n-1/2\right) \pi }{b-a}\right) ^{2}-K\right) <\infty  \tag{3.1}
\end{equation}%
converges and is called the regularized first trace for the problem
(1.1)-(1.5). The goal\ of this paper is to prove its sum. Our proof is based
on the works of [16,17].

\begin{theorem}
Suppose that $q\left( x\right) $ has a second-order piecewise integrable
derivatives on $\left[ a,b\right] ,~$then the following regularized trace
formula holds%
\begin{eqnarray}
s_{\lambda } &=&\sum\limits_{n=0}^{\infty }\left( \lambda _{n}-\left( \frac{%
\left( n-1/2\right) \pi }{b-a}\right) ^{2}-K\right)  \notag \\
&=&h-\frac{1}{2}-\frac{2}{b-a}-\frac{\pi ^{2}}{4\left( b-a\right) ^{2}}-%
\frac{1}{4}\left( q\left( b\right) -q\left( a\right) \right) -  \notag \\
&&\frac{1}{b-a}\left( Q_{1}\left( c_{1}\right) +Q_{2}\left( c_{2}\right)
+Q_{3}\left( b\right) \right) -  \notag \\
&&\frac{1}{8}\left( Q_{1}^{2}\left( c_{1}\right) +Q_{2}^{2}\left(
c_{2}\right) +Q_{3}^{2}\left( b\right) \right) ,  \TCItag{3.2}
\end{eqnarray}%
\textit{where }$K,\alpha _{1}\left( b\right) $ and $Q_{i}\left( c_{i}\right) 
$ $\left( i=1,2,3\right) $ satisfies the equations (2.14)-(2.16).\textit{\ }
\end{theorem}

\begin{proof}
...

\begin{lemma}
If $\left\vert \left( \frac{\left( n-1/2\right) \pi }{b-a}\right)
^{2}-\lambda _{n}\right\vert \leq \rho ,$ then 
\begin{equation}
\sum\limits_{n=1}^{\infty }\dfrac{\left\vert \left( \frac{\left(
n-1/2\right) \pi }{b-a}\right) ^{2}-\lambda _{n}\right\vert ^{k}}{\left( \mu
^{2}+\left( \frac{\left( n-1/2\right) \pi }{b-a}\right) ^{2}\right) ^{k}}%
\leq \rho ^{k}\frac{b-a}{2}\frac{1}{\mu ^{2k-1}}.  \tag{3.9}
\end{equation}
\end{lemma}
\end{proof}

\begin{proof}
This lemma can be proved similar to proof of Lemma 5.6.1 in $\left[ 17,Ch5%
\right] .$

Substituting (2.9), (3.8), (3.14) and (3.20) into (3.23) we obtain%
\begin{eqnarray*}
s_{\lambda } &=&\sum\limits_{n=0}^{\infty }\left( \lambda _{n}-\left( \frac{%
\left( n-1/2\right) \pi }{b-a}\right) ^{2}-K\right)  \\
&=&h-\frac{1}{2}-\frac{2}{b-a}-\frac{\pi ^{2}}{4\left( b-a\right) ^{2}}-%
\frac{1}{4}\left( q\left( b\right) -q\left( a\right) \right) - \\
&&\frac{1}{b-a}\left( Q_{1}\left( c_{1}\right) +Q_{2}\left( c_{2}\right)
+Q_{3}\left( b\right) \right) -\frac{1}{8}\left( Q_{1}^{2}\left(
c_{1}\right) +Q_{2}^{2}\left( c_{2}\right) +Q_{3}^{2}\left( b\right) \right) 
\end{eqnarray*}%
where%
\begin{equation*}
Q_{1}\left( c_{1}\right) =\underset{a}{\overset{c_{1}}{\int }}q\left(
y\right) dy,\ \ \ Q_{2}\left( c_{2}\right) =\underset{c_{1}}{\overset{c_{2}}{%
\int }}q\left( y\right) dy,\ \ Q_{3}\left( b\right) =\underset{c_{2}}{%
\overset{b}{\int }}q\left( y\right) dy
\end{equation*}%
completing the proof of Theorem 1.%
\begin{equation*}
\end{equation*}
\end{proof}

\end{document}